\documentclass[12pt]{amsart}
\usepackage[english]{babel}
\usepackage[utf8]{inputenc}
\usepackage{fullpage}
\usepackage{amsmath}
\usepackage{tikz-cd}
\newtheorem{defi}{Definition}[section]
\newtheorem{exa}{Example}[section]
\newtheorem{lema}{Lemma}[section]
\newtheorem{teo}{Theorem}[section]
\newtheorem{rem}{Remark}[section]
\newtheorem{coro}{Corollary}[section]
\newtheorem{pro}{Proposition}[section]

\newcommand{\T}{\mathbb{T}}
\newcommand{\Td}{\mathbb{T}^d}
\newcommand{\C}{\mathbb{C}}

\newcommand{\R}{\mathbb{R}}
\newcommand{\N}{\mathbb{N}}
\newcommand{\Z}{\mathbb{Z}}
\newcommand{\norm}[1]{\left\lVert#1\right\rVert}
\newcommand{\esp}{\text{  }}

\begin{document}

\title[ON SOME SPECTRAL PROPERTIES OF PDOS ON  $\mathbb{T}$]
 {ON SOME SPECTRAL PROPERTIES OF PSEUDO-DIFFERENTIAL OPERATORS ON  $\mathbb{T}$}

\author[J. P. Velasquez-Rodriguez]{Juan Pablo Velasquez-Rodriguez}

\address{Department of Mathematics\\
Universidad del Valle\\
Cali\\
Colombia}

\email{velasquez.juan@correounivalle.edu.co}




\date{August 10, 2018}
\begin{abstract}
In this paper we use Riesz spectral Theory and  Gershgorin Theory to obtain explicit information concerning the spectrum of pseudo-differential operators defined on the unit circle $\mathbb{T} := \mathbb{R}/ 2 \pi \mathbb{ Z}$. For symbols in the Hörmander class $S^m_{1 , 0} (\T \times \Z)$, we provide a  sufficient and necessary condition to ensure that the corresponding pseudo-differential operator is a Riesz operator in $L^p (\mathbb{T})$, $1< p < \infty$, extending in this way compact operators characterisation in \cite{shahla} and Ghoberg's lemma in \cite{Molahajloo2010} to $L^p (\mathbb{T})$. We provide an example of a non-compact Riesz pseudo-differential operator in $L^p (\T)$, $1<p<2$. Also, for  pseudo-differential operators with symbol satisfying some integrability condition, it is defined its associated matrix in terms of the Fourier coefficients of the symbol, and this matrix is used to give necessary and sufficient conditions for $L^2$-boundedness without assuming any regularity on the symbol, and to locate the spectrum of some operators. 
\end{abstract}

\maketitle

\section*{\textbf{INTRODUCTION}}

Pseudo-differential operators acting on functions defined on smooth manifolds are an important generalisation of differential operators, and their study is a very active branch of contemporary mathematics, among other reasons, because of its applications in various areas of physics, such as black-hole physics, quantum electrodynamics and quantum field theory \cite{QFT, Geometry}. Pseudo-differential operators considered when the manifold is a compact Lie group have arisen as in interesting arena from the point of view of Ruzhansky-Turunen Theory, in which the representations of the group allow to define a global symbol instead of symbols depending on charts. 

The simplest example in this direction is the case of the one dimensional torus, on which we have the concept of periodic pseudo-differential operators acting on periodic functions. These operators have been widely
studied \cite{JulioLpbounds, Duvan1, Duvan2, MolahajloS1, Ghaemi2, Ghaemi3, Ghaemi2017, Pirhayati2011, delgado_ruzhansky_2017} and remarkable results have been obtained, being one of the most important, perhaps the most, the global definition of pseudo-differential operators on the unit circle in \cite{Agranovich1979}, proposed by Agranovich in 1979 crediting L.R. Volevich.  The definition is illuminating and readily generalisable for any torus $\Td$, but it is a non-trivial fact that it is equivalent to the original one given by Hörmander. However, time after Agranovich published a definition, McLean in \cite{Mclean} proved the equivalence for operators with symbols in the Hörmander $(\rho , \delta)$- classes, and since then other proofs have been published.

In this paper we will be working within the framework of the global symbols introduced by Ruzhansky and Turunen with extension and periodisation techniques \cite{ruzhanskytor2008} since their approach is useful and elegant and that is why their book \cite{ruzhansky1} will be one of our main references. This work is organised as follows: 

\begin{itemize}
    \item In Section 1 we recall some of the basics on periodic pseudo-differential operators from the point of view of Ruzhansky and Turunen Theory.
\item In Section 2, we treat Riez operators in the sense of Pietsch, which are  those operators acting on a Banach space with spectrum behaviour identical to that of compact operators. Inspired by Molahajloo in \cite{shahla}, we give the main theorem of the section, a necessary and sufficient condition for a pseudo-differential operator with symbol in the Hörmander class $S^0_{1,0} (\mathbb{T} \times \mathbb{Z} )$ to be a Riesz operator in $L^p (\mathbb{T})$, $1< p < \infty$ (Theorem 2.4), together with an example of a non-compact Riesz pseudo-differential operator.

\item In Section 3 the interest is focused in the study of global pseudo-differential operators acting on $L^2 (\T)$. It is defined the infinite matrix associated to the operator, and it is used to study operators in $L^2 (\T)$ factorising in terms of operators in $L^2 (\Z)$, which are generally easier to deal with. As a result we obtain necessary and sufficient conditions for boundedness (Theorems 3.1 and 3.2) without assuming any regularity on the symbol, and sufficient conditions for invertibility, and spectrum localisation (Theorems 3.4 and 3.5).

\item In Section 4 we discuss a proof of Ghoberg's lemma in $L^p (\T)$ as an extension of \cite[Theorem 3.3]{Molahajloo2010}.

\end{itemize}

 \section{\textbf{ Periodic  Pseudo-differential Operators. }}
 In this section we recall the definition of  periodic  global pseudo-differential operator, the periodic version of  Hörmander classes and the basics for the corresponding pseudo-differential calculus in the sense of Ruzhansky-Turunen theory \cite{ruzhansky1}. We remark that it is possible to consider global pseudo-differential operators associated to full symbols which are $p$-integrable in the toroidal variable.
 
As usual for an integrable function $f: \T \to \C$ we denote its $k$-th Fourier coefficient by $\widehat{f} (k)$, which is defined by $$\widehat{f} (k) : = \int_\T f(x) e^{-i x \cdot k} dx,$$where $k \in \Z$, and $dx$ denotes the normalized Haar measure on $\T$. 
Let $\mathcal{P}(\T):= Span\{e^{i x \cdot k} \esp : \esp k \in \Z\}$ denote the collection of trigonometric polynomials on the unit circle. Consider a given linear operator $A: \mathcal{P} (\T) \to L^p (\T)$ for  $1< p < \infty $. Following  \cite[Theorem 10.4.6]{ruzhansky1} , we associate to the linear operator $A$ the {\em full symbol } $\sigma_A: \T \times \Z \to \mathbb{C}$, defined by $$\sigma_A (x , k) := e^{- i x \cdot k} (A e_k ) (x),$$ where $e_k (x):= e^{i x \cdot k}$. This full symbol satisfies that
\begin{align*}
    (A_\sigma f)(x) = \sum_{k \in \Z} \sigma_A (x , k) \widehat{f} (k) e^{i x \cdot k},
\end{align*}for all $f \in \mathcal{P}(\T)$. On the other hand, given a measurable function $\sigma : \T \times \Z \to \C$ we can define using the equality above a linear operator acting, initially, on some subspace of $L^1(\T).$ Our objective is to find suitable conditions on the function  $\sigma$ in order to ensure that the domain of the corresponding pseudo-differential operator $A_\sigma$ can be extended to  $L^p(\T)$. 
 
Let $m \in \R$, and $0 \leq \rho ,  \delta \leq 1$ given. The periodic Hörmander class $S^m_{\rho , \delta} (\T \times \Z)$ of symbols $\sigma \in C^\infty (\T \times \Z)$ is the set of functions satisfying the estimates
$$ |\Delta_{k}^{t} D_{x}^{r} \sigma (x, k)| \leq C_{t,r,\rho,\delta,m} \langle k \rangle^{m-\rho t + \delta r},$$for each $t,r \in \N$, $k \in \Z$, $x \in \T$,   $C_{t,r,\rho,\delta,m} \in \R^+ $, and $\langle k \rangle := (1+k^2)^{1/2}$. Here the difference operator $\Delta_k^t \varphi$ acts on functions $\varphi: \Z \to \C $ as $$\Delta_k^t \varphi (k) := \sum_{h=0}^{t} (-1)^{t-h} { {t}\choose{h} }\varphi (k + h),$$and for functions $f\in C^r (\T)$ the operator $D_x^r f$ is defined by $$D_x^r f (x) := (-i)^r \partial_x^r f (x).$$
Considering symbols in Hörmander classes allow us to treat with a graded $C^{*}$-algebra of pseudo-differential operators as a consequence of the pseudo-differential calculus summarised in \cite{ruzhansky1} .

We introduce the object of study in this paper, the periodic pseudo-differential operators with $p-$integrable symbols for $1<p< \infty$ fixed. This definition comes from the fact that, in order to obtain a well defined operator from $\mathcal{P}(\T)$ with the $L^p$-norm to $L^p (\T)$, the linear operator defined by formula 

$$(A_\sigma f)(x) = \sum_{k \in \Z} \sigma_A (x , k) \widehat{f} (k) e^{i x \cdot k},$$
should have symbol $\sigma_A (x , k)$ in $L^p (\T)$ for every fixed $k$.

\begin{defi}
{\normalfont Let  $1<p<\infty$ be given. Consider $\sigma : \T \times \Z \to \C $ be a measurable function such that $\sigma (\cdot , k) \in L^p (\T)$ for each $k \in \Z$. The {\em associated periodic pseudo-differential operator} $T_\sigma: \mathcal{P} (\T) \to L^p (\T)$ is defined by
$$(T_\sigma f)  (x) = \sum_{k \in \Z} \sigma (x , k) \widehat{f} (k) e^{i x \cdot k}.$$

In particular, if $\sigma (x,k):= \sigma(k)$ is a function depending only on $k$, the pseudo-differential operator $T_\sigma$ is called a Fourier multiplier.} 
\end{defi}

Recall that for symbols in Hörmander classes, it is possible to provide a symbolic calculus to periodic pseudo-differential operators.  This is stated formally in the following propositions \cite{ruzhansky1}.

\begin{pro}[Toroidal Composition Formula] 
Let $T_\alpha$, $T_\beta$ be periodic pseudo-differential operators with symbols $\alpha \in S^m_{\rho , \delta} (\T \times \Z), \beta \in S^l_{\rho , \delta} (\T \times \Z)$. Then $T_\alpha T_\beta$ is a pseudo-differential operator with symbol $ \sigma $ in the Hörmander class $ S^{m+l}_{\rho , \delta} (\T \times \Z)$ and the toroidal symbol $\sigma$ has the following asymptotic expansion:
$$ \sigma (x , k ) \sim \sum_{h}  \frac{1}{h !} \Delta_{k}^{h} \alpha (x, k) D_{x}^{h} \beta (x ,k) ,
$$where the above means that for each $N \in \N$ we have
 
$$\sigma (x , k ) - \sum_{h< N}  \frac{1}{h !} \Delta_{k}^{h} \alpha (x, k) D_{x}^{h} \beta (x ,k) \in S^{m-N}_{\rho,p \delta} (\T \times \Z).$$
\end{pro}

\begin{pro}[Adjoint Operator Formula.]
Let $T_\sigma$ be a pseudo-differential operator with symbol $\sigma \in S^m_{\rho , \delta} (\T \times \Z)$. Then the adjoint operator $(T_\sigma)^*$ is a pseudo-differential operator with symbol $\sigma^* \in S^m_{\rho , \delta} (\T \times \Z)$ and it has the following asymptotic expansion

$$\sigma^* \sim \sum_h \frac{1}{h!} \Delta_k^h \partial_x^h \overline{\sigma (x,k)}.$$
\end{pro}

\section{\textbf{Riesz Pseudo-differential Operators }} 
 
In this section we extend the main result of \cite{shahla} on the compactness of periodic pseudo-differential operators. We will provide an alternative proof of main Theorem  in that paper, which characterises the  compact pseudo-differential operators on $L^2(\T)$, using different arguments not requiring the existence of an inner product, implying a result, the main result of this section, for the $L^p (\T)$ case, where $1<p<\infty$. For this purpose we review the already known Gohberg's Lemma, wich we extend to $L^p (\T)$ in Section 4. In what follows,  $\mathcal{L} (E,F)$ will denote the collection of all bounded linear operators from $E$ to $F$ normed spaces, and $\mathfrak{K} (E,F)$ denotes the collection of compact operators in $\mathcal{L} (E,F)$. For $E=F$  we write $\mathcal{L}(E)$ instead of $\mathcal{L} (E,E)$. For $T \in \mathcal{L}(E)$ the resolvent set of $T$ will be denoted by $$Res(T):\{\lambda\in \mathbb{C}:= (T-\lambda I)^{-1}\in \mathcal{L}(E) \},$$ and the spectrum by  $Spec(T):=\mathbb{C}\setminus Res(T)$.

\begin{teo}[Gohberg's Lemma]
Let $T_\sigma$ be a pseudo-differential operator with symbol $\sigma \in S^0_{1,0} (\T \times \Z)$, and take $1<p<\infty$ fixed. Then $\norm{T_\sigma - K}_{\mathcal{L} (L^p (\T))} \geq d_\sigma $  for any compact operator $K \in \mathfrak{K}(L^p(\T))$, where    
\begin{align*}
     d_\sigma:= \limsup_{|k| \to \infty} \{\sup_{x \in \T} |\sigma (x , k)|\}.
\end{align*}
\end{teo}

A proof of the theorem   above can be found in \cite{Molahajloo2010}, where the authors study the case $p=2$. We assert that the proof also works if the $L^2$-norm is replaced by $L^p$-norm in each step. This statement will be proved in Section 4.

\begin{teo}[Molahajloo]\label{Molah}
Let $T_\sigma$ be a pseudo-differential operator with symbol $\sigma \in S^0_{1,0} (\T \times \Z)$. Then $T_\sigma$ is a compact  operator on $L^2 (\T)$ if and only if

\begin{align*}
     d_\sigma:&= \limsup_{|k| \to \infty} \{ \sup_{x \in \T} |\sigma (x , k)|\} =\lim_{|k| \to \infty} \{ \sup_{x \in \T} |\sigma (x , k)| \} = 0.
\end{align*}
\end{teo}

\begin{proof}
Assume that $d_\sigma =0$ and let $f \in C^2 (\T)$. For all $x \in \T$

\begin{align*}
    (T_\sigma f) (x) &= \sum_{k \in \Z} \sigma(x , k) \widehat{f}(k) e^{i x \cdot k} \\ &= \sum_{k \in \Z} \Big( \sum_{m \in \Z} \widehat{\sigma} (m , k) e^{i x \cdot m}\Big) \widehat{f} (k) e^{i x \cdot k} \\
    &= \sum_{m \in \Z} e^{i x \cdot m} \Big( \sum_{k \in \Z} \widehat{\sigma} (m , k) \widehat{f} (k) e^{i x \cdot k} \Big)  \\
    &= \sum_{m \in \Z}  e^{i x \cdot m} (T_{\widehat{\sigma}_m} f) (x),
\end{align*}where $\widehat{\sigma}_m (k) := \widehat{\sigma}(m,k)$ and the change in the order of summation is justified by Fubini–Tonelli's theorem since 

\begin{align*}
    \sum_{k \in \Z} \sum_{m \in \Z} |\widehat{\sigma} (m , k)| |\widehat{f}(k)| &= \sum_{k \in \Z} \norm{\widehat{\sigma}(\cdot , k)}_{L^1 (\Z)}  \cdot |\widehat{f}(k) |\\
    & \leq C_{0,2,0,0,0} \sum_{m \in \Z} \langle m \rangle^{-2} \cdot \sum_{k \in \Z}  |\widehat{f}(k) | < \infty .
\end{align*}

By defining the operator $(A_m f) (x) := e^{i x \cdot m} f (x)$, a multiplication operator, we have

$$(T_\sigma f) (x) = \sum_{m \in \Z}  (A_m  T_{\widehat{\sigma}_m} f) (x),$$
and clearly $A_m \in \mathcal{L} (L^2(\T))$. Now, for each $m \in \Z$, the operator $T_{\widehat{\sigma}_m}$ is a Fourier multiplier. As it is well known, a  pseudo-differential operator with symbol depending only on the Fourier variable is a Fourier multiplier which is a compact operator in $L^2 (\T)$ if and only if its symbol $\eta$ satisfy 

    $$\lim_{|k| \to \infty} |\eta (k)|=0.$$

Now, for each $m \in \Z$ we have that
$$
    \lim_{|k| \to \infty} |\widehat{\sigma} (m,k)| \leq \lim_{|k| \to \infty} \sup_{x \in \T} |\sigma (x , k)| \leq \limsup_{|k| \to \infty} \sup_{x \in \T} |\sigma (x,k)|= 0.$$

This implies that each operator $T_{\widehat{\sigma}_m}$ is a compact operator. As a consequence each  $A_m  T_{\widehat{\sigma}_m}$ is compact and for all $N \in \N$ , the operator
$$\displaystyle \sum_{|m| \leq N} A_m  T_{\widehat{\sigma}_m},$$
is also compact since the set of  compact operators $\mathfrak{K} (L^2 (\T))$ form a two sided ideal in $\mathcal{L} (L^2 (\T))$ (see \cite[Proposition 4.3.4]{vitali}) and this ideal of compact operators is a closed subset of $\mathcal{L} (L^2(\T))$ in the operator norm topology. For this reason, if the series

\begin{align*}
     \sum_{m \in \Z}  A_m T_{\widehat{\sigma}_m},
\end{align*}converge in the operator norm topology then 

\begin{align*}
    T_\sigma = \lim_{N \to \infty} \sum_{|k| \leq N} A_m T_{\widehat{\sigma}_m},
\end{align*}is compact as  it is the limit of a sequence of compact operators. 
 
 \begin{rem}This argument is the most important part of the proof because it shows that it is not necessary to appeal to inner product or Calkin's algebra to characterise compact operators. Moreover, it exposes, similar to Schatten classes \cite{schaten1DELGADO2014779, schaten2}, a relationship between operator ideals, approximation properties and pseudo-differential operators. In fact Theorem 2.1, seen from the perspective of operator ideals theory, is a characterisation of belonging to a certain closed ideal in terms of the symbol. This idea will be exploited later.
 \end{rem}

To conclude our proof, it is well known (See \cite[Proposition 2.4]{Molahajloo2010}) that if $\sigma \in S^0_{1,0} (\T \times \Z)$ then 

$$\sum_{m \in \Z} \norm{A_m  T_{\widehat{\sigma}_m}}_{\mathcal{L}(L^2(\T))} =  \sum_{m \in \Z} \norm{T_{\widehat{\sigma}_m}}_{\mathcal{L}(L^2(\T))} < \infty.$$

In summary, $T_\sigma$ is a compact operator. Now, assume that $d_\sigma \neq 0$. We need only to show that $T_\sigma$ is not compact on $L^2 (\T)$. Suppose that $T_\sigma$ is compact. If we set $T_\sigma = K$ in Theorem 2.1 then it contradicts our assumption that $d_\sigma \neq 0$.
\end{proof}

\subsection{Riesz Periodic Fourier Multipliers}
The previous proof of Theorem \ref{Molah} essentially uses two facts:  First,  Gohber's lemma which allow us to proof one  implication. Second, the set  of compact operators in $\mathcal{L} (L^2(\T))$ is a closed operator ideal. 
 In the case $L^p (\T)$ $1<p<\infty$ we will use the same arguments as in Theorem \ref{Molah} but, instead of the ideal of compact operators, we will use its radical.

We want to recall some basic definitions for the reader.
\begin{defi}
{\normalfont Let $\mathcal{L}$ be the collection of all bounded operators defined on Banach spaces. That is

$$\mathcal{L} := \bigcup_{E,F \text{ Banach Spaces}} \mathcal{L} (E,F).$$

An operator ideal is a sub-collection $\mathfrak{I}$ of $\mathcal{L}$ such that each one of its components

$$\mathfrak{I} (E,F) := \mathfrak{I} \cap \mathcal{L}, $$
satisfy the following conditions:

\begin{enumerate}
    \item[(i)] For each one-dimensional Banach space $E$ the identity map $I_E$ belongs to $\mathfrak{I}(E)$.
    \item[(ii)] If $T_1 , T_2 \in \mathfrak{I} (E,F)$ then $T_1 + T_2 \in \mathfrak{I}(E,F)$.
    \item[(iii)] If $T \in \mathfrak{I}(E,F)$ and $X\in \mathcal{L}(F,F_0) , Y \in \mathcal{L} (E_0 , E)$ then their composition $X  T Y \in \mathfrak{I} (E_0 , F_0)$, where $E_0 $ and $F_0$ are Banach spaces.
    \end{enumerate}
    
If additionally each component $\mathfrak{I} (E,F)$ is closed in the operator norm topology of $\mathcal{L} (E,F)$ then it is said that $\mathfrak{I}$ is a closed operator ideal.}     
\end{defi}
 Given a operator ideal $\mathfrak{I}$ it is said that $S \in \mathcal{L} (E,F)$ belongs to the radical of $\mathfrak{I}$ (see \cite[4.3.1]{pietsch2}) if for all $L \in \mathcal{L} (F,E)$ there exist $U \in \mathcal{L} (E)$ and $X \in \mathfrak{I} (E)$ such that $$U(I_E - LS) = I_E -X.$$ The collection of operators that satisfy this condition is denoted by $\mathfrak{I}^{rad}$. It can be proved that $\mathfrak{I}^{rad}$ is a closed operator ideal so, to take advantage of this fact, we will work from now on with bounded operators in the radical of an special ideal. The radical of compact operators ideal is called the ideal of Gohberg or inessential operators, and is completely characterised by  \cite[Theorem 26.7.2]{pietsch2}.

\begin{teo}[Pietsch]
$\mathfrak{R} := \mathfrak{K}^{rad}$ is the largest operator ideal such that all components $\mathfrak{R} (E)$ consists of Riesz operators only.
\end{teo}
 
Riesz operators are a generalisation of compact operators that has been widely studied (\cite{calkin, pietsch, West, Ruston} and for that reason there are many equivalent definitions. One can use, for example, the definition given in \cite{pietsch}.

\begin{defi}
{\normalfont Let $E$ be a Banach space. It is said that $T \in  \mathcal{L} (E)$ is a {\em Riesz operator} if for all $\varepsilon > 0$ exist an exponent $s$ and points $v_1 ,..., v_h \in E$ which depend on $\varepsilon$ such that

$$T^s (B_E) \subseteq \bigcup_{j = 1}^{h} v_j + \varepsilon^s B_E,$$ where $B_E$ is the unit ball of $E$.}
\end{defi}

This definition does not provides too much information at first sight, but fortunately one can think of Riesz operators in a more intuitive way. In simple terms a bounded linear operator is a Riesz operator if its spectrum is countable, it contains  only eigenvalues, each one of these has finite-dimensional associated eigen-space, and the number zero is the only possible accumulation point. In other words, Riesz operators are those bounded linear operators whose spectrum behaves like the spectrum of compact operators. They are generally easier to study than the compact operators because of their different characterisations. Several necessary and  sufficient conditions to be a Riesz operator are listed in the following lemma, which is the collection of Propositions 3.2.13 and 3.2.24 in \cite{pietsch}, 26.5.1 in \cite{pietsch2}, and the corollary of Theorem 3.4.3 in \cite{calkin}. 
\begin{lema}
Let $E$ be a Banach space and suppose $T \in \mathcal{L} (E)$. The following statements are equivalent
 
\vspace{2mm}
\begin{enumerate}
\item[(i)] $T$ is a Riesz operator.
\item[(ii)] $T^m$ is Riesz for some (for all) exponent $m$.
\item[(iii)] $\lambda T$ is iterative compact for all $\lambda \in \C$.
\item[(iv)] Non-zero points in $Spec(T)$ are isolated eigenvalues with finite-dimensional associated eigenspace and zero as the only cluster point. If $E$ is infinite dimensional then $0 \in Spec (T)$. 
\item[(v)] $T- \lambda I$ is a Fredholm operator for all $\lambda \in \C$.
\end{enumerate}
\end{lema}

In general, Riesz operators do not form an operator ideal. However, what we want to show now is that in certain subalgebras of $\mathcal{L} (L^p(\T))$, $1<p<\infty$, there are ideals that only contain Riesz operators, and this will be a crucial fact in this work. The following two propositions are the first step in that direction.
\esp
\vspace{4mm}
\esp

\begin{pro}
Let $\sigma : \Z \to \C$ be a function such that
\begin{align}
    |\Delta_k^t \sigma (k)| \leq C \langle k \rangle^{-t} \esp , \esp k\in \Z
\end{align}
for $0 \leq t \leq [1/2] + 1$. Then for $1<p< \infty$ the spectrum of its associated Fourier multiplier $T_\sigma$ as an operator in $\mathcal{L} (L^p (\T))$ is:
$$Spec(T_\sigma)=\overline{ \{\sigma (k) \esp : \esp k \in \Z\}}$$
\end{pro}

\begin{proof}
First, the condition $(1)$ guarantees that $T_\sigma$ satisfy the conditions of the toroidal version of Mikhlin multiplier Theorem (see \cite{wongdfa}, Lemma 22.12) then $T_\sigma \in \mathcal{L} (L^p(\T))$. Define
$$\sigma_\lambda (k):= \sigma(k) - \lambda \esp \esp  \text{and} \esp  \esp \sigma_\lambda^{-1} (k) := \frac{1}{\sigma_\lambda (k)}$$for $\lambda \in \C \setminus \overline{ \{\sigma (k) \esp : \esp k \in \Z\}} $. Since $\C$ is a regular topological space, there exist a $d>0$ such that

\begin{align*}
    \inf_{k \in \Z} |\sigma(k) - \lambda| = d.
\end{align*}

Then $T_{\sigma_\lambda^{-1}}$ also satisfy the conditions of Mikhlin multiplier theorem because 

\begin{align*}
    | \Delta_k \sigma_\lambda^{-1} (k) |&= \Big| \frac{1}{\sigma_\lambda (k+1)} - \frac{1}{\sigma_\lambda (k)}\Big| \\ &= \Big| \frac{\sigma (k+1) - \sigma (k)}{\sigma_\lambda (k+1) \sigma_\lambda (k)} \Big| \leq \frac{C}{d^2} \langle k \rangle^{-1}.
    \end{align*}
    
Finally, it is clear that
$$(T_\sigma - \lambda I) (T_{\sigma_\lambda^{-1}}) = (T_{\sigma_\lambda}) (T_{\sigma_\lambda^{-1}}) =I,$$which proves that $Spec(T) \subseteq \overline{ \{\sigma (k) \esp : \esp k \in \Z\}}$. Now its easy to see that $\overline{ \{\sigma (k) \esp : \esp k \in \Z\}} \subseteq Spec(T)$ because the spectrum is closed and each number $\sigma (k)$ is an eigenvalue of $T_\sigma$ with associated eigenvector $e^{i x \cdot k}$, thus
$$\{\sigma (k) \esp : \esp k \in \Z\} \subseteq Spec(T).$$
\end{proof}

As an straightforward  consequence we obtain the following. 

\begin{pro}
Let $T_\sigma : L^p (\T) \to L^p (\T)$ $1<p<\infty$ be a Fourier multiplier whose symbol satisfy the condition $(1)$. Then $T_\sigma$ is a Riesz operator if and only if

$$\lim_{|k| \to \infty} |\sigma (k)| = 0.$$
\end{pro}

\begin{proof}
By the previous proposition and $(iv)$ in Lemma 2.1. 
\end{proof}
 
\begin{rem}[Algebraic considerations, part 1]
 
\esp
\begin{enumerate}
\item[(i)] The collection of Fourier multipliers
$$M_p := \{T_\sigma: L^p (\T) \to L^p (\T) \esp : \esp \sigma \esp \text{ satisfy (1)  } \esp \},$$
is a sub-algebra of $\mathcal{L}(L^p(\T))$, $1 < p<\infty$. This is a consequence of discrete Leibniz rule.

\item[(ii)] If $\mathcal{A}$ is an algebra and $\mathcal{B}$  is a sub-algebra of $\mathcal{A}$ then each ideal $\mathcal{I}$ of $\mathcal{A}$ defines an ideal of $\mathcal{B}$, this is $\mathcal{I} \cap \mathcal{B}$.

\item[(iii)] The component $\mathfrak{R} (L^p(\T))$ is the largest ideal in $\mathcal{L} (L^p (\T))$ which consists only of Riesz operators. As an immediate consequence $\mathfrak{R} (L^p(\T)) \cap M_p$ is the largest ideal on $M_p$ that consists only of Riesz operators.
\item[(iv)] The set
$$\mathfrak{R} (M_p):= \{ T_\sigma \in M_p \esp : \esp \lim_{|k| \to \infty} |\sigma (k)| = 0 \},$$
is an ideal of $M_p$, which,  in virtue of Proposition 2.2, contains only Riesz operators, so $\mathfrak{R} (M_p) \subseteq \mathfrak{R} \cap M_p$.

\item[(v)] Again, by Proposition 2.2, the symbol of every Fourier multiplier  $T_\sigma \in \mathfrak{R} \cap M_p$ satisfy $\lim_{|k| \to \infty} |\sigma (k)| = 0$, so $\mathfrak{R} \cap M_p \subseteq \mathfrak{R} (M_p)$ and in conclusion $\mathfrak{R} \cap M_p = \mathfrak{R} (M_p)$, that is, all Riesz Fourier multipliers are inessential operators.
\end{enumerate}

\end{rem}

\subsection{Riesz Periodic Pseudo-differential Operators}
We have constructed enough tools to prove the main theorem of this section. We will use the same  techniques as in Theorem 2.2 but  instead of compact operators we will use inessential operators.

\begin{teo}
Let $T_\sigma$ be a pseudo-differential operator with symbol $\sigma \in S^0_{1,0} (\T \times \Z)$ and $1<p<\infty$. Then $T_\sigma$ is a Riesz operator on $L^p (\T)$ if and only if
$$d_\sigma' := \lim_{|k| \to \infty} \{ \sup_{x \in \T} |\sigma (x,k)| \} =0.
$$
\end{teo}

\begin{proof}
We assert  that a pseudo-differential operator with  symbol in $S^0_{1,0} (\T \times \Z)$ is a Riesz operator if and only if $d_{\sigma^n}' = 0$ for some $n \in \N$, where $\sigma^n$ denotes the symbol of $(T_\sigma)^n$. Then to finish we will show that $d_{\sigma^n}' = 0$ if and only if $d_\sigma' = 0$.

\esp
First, suppose that $T_\sigma$ is a Riesz operator and $d_{\sigma^n}' > 0$ for all $n$. Using $(iii)$ in Lemma 2.1 we know that, for some  $h \in \N$,  $(T_\sigma)^h$ is a compact operator. Then, by Gohberg's lemma, we have $$0 < d_{\sigma^h}' \leq d_{\sigma^h} \leq \norm{(T_\sigma)^h - K}_{\mathcal{L} (L^p (\T))},$$
for all compact operators $ K $. If we set $K = (T_\sigma)^h$, the above chain of inequalities contradicts our assumption that $d_{\sigma^h}' \neq 0$.

Now suppose $d_{\sigma^n}' = 0$ for some $n \in \N$. Then for $f \in C^2 (\T)$

\begin{align*}
(T_{\sigma^n} f) (x) &= \sum_{k \in \Z} {\sigma^n (x , k)} \widehat{f} (k) e^{ix \cdot k} \\
&= \sum_{k \in \Z} \Big( \sum_{m \in \Z} {\widehat{\sigma}^n}(m ,k) e^{i x \cdot m} \Big)  \widehat{f} (k) e^{ix \cdot k} \\
&=\sum_{m \in \Z}  (A_m  T_{\widehat{\sigma}_m^n} f) (x),
\end{align*}and $$\lim_{|k| \to \infty} |\widehat{\sigma}^n (m,k)| \leq \lim_{|k| \to \infty} \sup_{x \in \T} |\sigma^n (x , k)|= 0,$$
so, for each $m \in \Z$, $T_{\sigma^n_m}$ is an operator in the ideal $\mathfrak{R} (M_p) $ which is a subset of the closed ideal $\mathfrak{R} (L^p (\T))$ of inessential operators. As a consequence each operator $A_m  T_{\widehat{\sigma}_m^n}$ belongs to  $\mathfrak{R} (L^p (\T))$ and for all $N \in \N$ the operator $$\sum_{|m| \leq N} A_m  T_{\widehat{\sigma}_m^n},$$
belongs to $\mathfrak{R} (L^p (\T))$ as well. As already said, the series $$\sum_{m \in \Z}  A_m T_{\widehat{\sigma}_m^n},$$ 
converge in the operator norm topology. In summary $T_{\sigma^n}$ belongs to $\mathfrak{R} (L^p (\T))$ as the limit of a sequence of inessential operators
$$T_{\sigma^n} = \lim_{N \to \infty} \sum_{|m| \leq N} A_m T_{\widehat{\sigma}_m^n}.$$
Finally, let us see that $d_{\sigma^n}' = 0$ if and only if $d_\sigma'=0$. This is a simple consequence of toroidal composition formula. As we know, by toroidal composition formula,  taking the first two terms in the asymptotic expansion of $T_{\sigma^n} =T_\sigma  T_{\sigma^{n-1}}$ we have
$$\sigma^{n} (x , k) = (\sigma(x, k))(\sigma^{n-1} (x , k)) + r_{n}  (x, k),$$with $r_{n}  (x, k)  \in S^{-1}_{1,0} (\T \times \Z)$. Repeating this process we obtain

\begin{align*}
     \sigma^{n} (x , k) &= (\sigma(x, k ))^2 (\sigma^{n-2} (x , k)) + (\sigma (x , k))(r_{n-1} (x , k)) + r_{n} (x , k) \\ &= ... \\ &= (\sigma(x , k))^n +  \sum_{h=0 }^{n-1} (\sigma (x, k))^h (r_{n-h} (x, k)),
\end{align*}with each $r_j (x, k)  \in S^{-1}_{1,0} (\T \times \Z)$. From this we have
$$(\sigma(x , k))^n =  \sigma^{n} (x , k)  -  \sum_{h=0 }^{n-1} (\sigma (x, k))^h (r_{n-h} (x, k)),$$which implies $$|\sigma(x , k)|^n \leq   |\sigma^{n} (x , k)|  +  \sum_{h=0 }^{n-1} |\sigma (x, k)|^h |r_{n-h} (x, k)|,$$ and $$\sup_{x \in \T} |\sigma(x , k)|^n  \leq   \sup_{x \in \T} |\sigma^{n} (x , k)| +  \sum_{k=0 }^{n-1} C_{\sigma}^h \sup_{x \in \T} |r_{n-k} (x, k)|,$$

where $C_\sigma$ is a constant that bounds uniformly $|\sigma (x , k)|$ and whose existence is clear because $\sigma \in S^{0}_{1,0} (\T \times \Z)$. In conclusion $$\lim_{|k| \to \infty} \sup_{x \in \T} |\sigma(x , k)|^n  \leq   \lim_{|k| \to \infty} \sup_{x \in \T} |\sigma^{n} (x , k)| +  \sum_{h=0 }^{n-1} \lim_{|k| \to \infty} C_{\sigma}^h \sup_{x \in \T} |r_{n-h} (x, k)|.$$

This finish the proof.
\end{proof}

\begin{rem}[Algebraic considerations, part 2.]
\esp
\begin{enumerate}
\item[(i)] The collection of pseudo-differential operators
$$\Psi_p^m := \{T_\sigma : L^p (\T) \to L^p (\T) \esp : \esp \sigma \in S^m_{1,0} (\T \times \Z) \},$$
is a sub-algebra of $\mathcal{L}(L^p (\T))$, this as a consequence of the toroidal composition formula.
\item[(ii)] The component $\mathfrak{R} (L^p(\T))$ is the largest ideal in the algebra $\mathcal{L}(L^p (\T))$ consisting only of Riesz operators. As an immediate consequence $ \mathfrak{R}  \cap \Psi_p^m$ is the largest ideal of $\Psi_p^m$ consisting only of Riesz operators
\item[(iii)] The set $$ \text{  } \text{  } \text{  } \text{  } \esp \esp \text{  } \mathfrak{R} (\Psi_p^m) := \{T_\sigma : L^p(\T) \to L^p(\T) \esp : \esp \lim_{|k| \to \infty} \{ \sup_{x \in \T} |\sigma (x , k)| \} = 0 \},$$
is an ideal of $\Psi_p^m$ that, by the previous theorem , consist only of Riesz operators. Then $\mathfrak{R} (\Psi_p^m) \subseteq \mathfrak{R} \cap \Psi_p^m$.
\item[(iv).] Again, by the previous theorem, the symbol of any linear operator  $T_\sigma \in \mathfrak{R} \cap \Psi_p^m$ satisfy $$\lim_{|k| \to \infty} \{ \sup_{x \in \T} |\sigma (x , k)| \} = 0.$$

In conclusion $\mathfrak{R} (\Psi_p^m) = \mathfrak{R} \cap \Psi_p^m$.
\end{enumerate}

\end{rem}

We finish this section with a remark about compact operators on $L^p (\T)$. The statements of Proposition 2.2 and Theorem 2.4 could be more general. With sufficient conditions on the symbol for the boundedness of Fourier multipliers or pseudo-differential operators, let us call this conditions $(BF)$ and $(BP)$ respectively, and similar proofs as in Proposition 2.2 and Theorem 2.4, one can state: 

\esp

\begin{center}

\begin{minipage}{0.9\textwidth}

Let $T_\sigma : L^p (\T) \to L^p (\T)$ $1<p<\infty$ be a Fourier multiplier whose symbol $\sigma$ satisfy the condition $(BF)$. Then $T_\sigma$ is a Riesz operator if and only if $$\lim_{|k| \to \infty} |\sigma (k)| = 0.$$
Let $T_\sigma:L^p (\T) \to L^p (\T)$ be a pseudo-differential operator whose symbol $\sigma$ satisfy the condition $(BP)$. Then $T_\sigma$ is a Riesz operator in $L^p (\T)$ if and only if
$$d_\sigma' := \lim_{|k| \to \infty} \{ \sup_{x \in \T} |\sigma (x,k)| \} =0.
$$
\end{minipage}
\end{center}

Unfortunately conditions $(BF)$ and $(BP)$ are not easy to find, but we emphasise this more general form of Proposition 2.2 and Theorem 2.4 to give an example of a Riesz non-compact pseudo-differential operators, showing then that the problem of find necessary an sufficient conditions on the symbol for the compactness of the operator in $L^p (\T)$ is still open. For this, we will exhibit a bounded pseudo-differential operator whose symbol tends to zero uniformly, and thus is a Riesz operator, but not a compact operator on $L^p (\T)$.

As it is shown in \cite{ex}, there exist strictly singular operators on $L^p [0,1]$, and then on $L^p (\T)$, which are non-compact operators. Strictly singular operators are  bounded linear operators between normed spaces which are not bounded below on any infinite-dimensional subspace, and the class of strictly singular operators is contained in the class of Riesz operators on any Banach space.  For $L^p (\T)$, $1 <p<2$, one can construct an example of a strictly singular non-compact operator through the following process: consider a complemented subspace $F_p$ of $L^p (\T)$ isomorphic to $\ell^p(\N)$, generated by disjointly supported functions, and denote by $P_p$ a projection
from $L^p (\T)$ onto $F_p$. Take the inclusion $i_{ p,2}: \ell^p(\N) \to \ell^2 (\N)$, and define the operator $Q: \ell^2 (\N) \to L^p (\T)$ by $$Q[(t_n)](x) = \sum_{n \in \N } t_n r_n (x),$$ where $r_n$ are the Rademacher functions $r_n (x) = sign(\sin (2^{k-1} x))$. By Khintchine’s inequality, the operator $Q$ is an isomorphic embedding of $\ell^2(\N)$ into $L^p(\T)$ for every $1 < p < \infty$, and the operator $A _p : L^p (\T) \to L^p (\T)$ given by $$A = Q i_{p,2} P_{p},$$ is a non-compact strictly singular operator \cite{ex}. To construct our example let us define $$E_n := (\frac{2\pi}{2^{n+2}}, \frac{2 \pi}{2^{n+1}}), \esp \esp n \in \N,$$ and let $\chi_{n} (x)$ be the characteristic function on $E_n$. Take the closed subspace $F_p$ as $$F_p := Span \{2^{(n+2)/2} \chi_n \}_{n \in \N},$$ which, as it is discussed in \cite[Chapter 3, p.p. 126]{handbook1}, is isomorphic to $\ell^p (\N)$ and 1-complemented with the projection $P_p$ given by $$P_p [f](x) = \sum_{n \in \N} \Big( \int_{\T} 2^{(n+2)(p-1)/2}  \chi_n (x) \overline{f(x)} dx\Big) 2^{(n+2)/2} \chi_{n} (x).$$  Thus the linear operator $A = Q i_{p,2} P_{p}$ is a pseudo-differential operator with symbol 
\begin{align*}
\sigma_A (x,k) = e^{-i x \cdot k} A (e^{i x \cdot k}) &= e^{-i x \cdot k} \sum_{n \in \N} \Big( \int_{\T} 2^{(n+2)(p-1)/2}  \chi_n (x) e^{-i x \cdot k} dx\Big) r_n (x) \\ &= e^{-i x \cdot k} \sum_{n \in \N} \Big( \int_{2 \pi/2^{n+2}}^{2 \pi/2^{n+1}} 2^{(n+2)(p-1)/2}  e^{-i x \cdot k} dx\Big)  r_n (x) \\ &= \frac{e^{- i x \cdot k}}{ik} \sum_{n \in \N} 2^{(n+2)(p-1)/2} e^{-i k \frac{2 \pi}{2^{n+2}}} (e^{-i k \frac{2 \pi}{2^{n+2}}} - 1)r_n (x). 
\end{align*}

We note that each $k \in \Z$ fixed, and $n \in \N $ large enough, we get 
\begin{align*}
|e^{-ik \frac{2 \pi}{2^{n+2}}} - 1|^2 &= 2|\cos{\frac{2 \pi k}{2^{n+2}}} - 1| = 2|\cos{\frac{2 \pi k + 0}{2^{n+2}}} - \cos{0}|\\ &= 4 |\sin{\frac{2 \pi k - 0}{2^{n+3}}} \sin{\frac{2 \pi k + 0}{2^{n+3}}}| \\ &\leq \frac{ \pi |k| }{2^{n+2}},
\end{align*}
which proves that there exist a constant $C>0$ such that $$S_k:=\sum_{n \in \N}  2^{(n+2)(p-1)/2} |e^{ik \frac{2 \pi}{2^{n+2}}} - 1| \leq C |k|^{1/2},$$ and then 

$$\lim_{|k| \to \infty} \sup_{x \in \T} |\sigma_A (x , k)| \leq \lim_{|k| \to \infty} \frac{S_k}{|k|} \leq \lim_{|k| \to \infty} \frac{C}{|k|^{1/2}} = 0,$$ so $A$ is a non-compact Riesz operator.
\section{\textbf{Gershgorin Theory for Periodic Pseudo-differential Operators}}

In this section we will leave aside the spaces $L^p (\T)$, $p \neq 2$, and focus in the Hilbert space $L^2(\T)$. The main reason is that the Fourier transform is an isomorphism between $L^2(\T)$ and $L^2 (\Z)$ which allows to factorise operators in $L^2 (\T)$ through $L^2(\Z)$ a sequence space that, in some cases, it is easier to study. This factorisation allows one to obtain necessary and sufficient conditions for boundedness (Theorems 3.1 and 3.2) without assuming any regularity on the symbol, and sufficient conditions for invertibility, and spectrum localisation (Theorems 3.4 and 3.5). To start with , it is necessary to mention some details about the infinite matrices theory.

\subsection{Infinite Matrices}
\begin{defi}
{\normalfont An infinite matrix is a function $M: \Z \times \Z \to \C$ with matrix entries defined by $M_{jk} := M(j,k)$. If $M$ is an infinite matrix and $\varphi$ an infinite vector (or a function from $\Z$ to $\C$) then the product of the vector $\varphi$ an the matrix $M$ is defined as $$M\varphi (j) := \sum_{k \in \Z} M_{jk} \varphi (k).$$For infinite matrices $P$ and $Q$ their product is defined as the infinite matrix with entries $$ PQ_{jk} = \sum_{h \in \Z} P_{jh} Q_{hk},$$and as usual, the adjoint of the infinite matrix $M$ is the infinite matrix $M^*$ with entries $$M^*_{jk} := \overline{(M_{kj})}.$$
}
\end{defi}
It is easy to see that with the above definition for any pair of infinite vectors (functions $\varphi_1, \varphi_2 : \Z \to \C$) and complex numbers $\lambda_1 , \lambda_2$ one has $$M(\lambda_1 \varphi_1 + \lambda_2 \varphi_2)  = \lambda_1 M \varphi_1  + \lambda_2 M \varphi_2 ,$$So it is reasonable to think that an infinite matrix $ M $ can define a linear operator on some sequence space. However, not all infinite matrices define linear operators, and some conditions should be imposed on the matrix to be sufficiently well behaved to define a linear operator. In this case we are interested in linear operators on $L^2 (\Z)$. Fortunately infinite matrices that define operators in $ L^2 (\Z) $ have already been studied and many results have been obtained. We state the most relevant for our work below. They can be found in \cite{Crone1971}.

\begin{lema}[Crone]
Let $M$ be an infinite matrix with rows and columns in $L^2(\Z)$ and define $P_n(x) := \sum_{|k| \leq n} \langle x , e_k \rangle e_k $, where $e_k$ is the $k$-th coordinate vector of $L^2 (\Z)$. Then $M$ define a bounded operator in $L^2 (\Z)$ if and only if
$$\sup_{n \in \N} \norm {P_n  M^*  M  P_n}_{\mathcal{L}(L^2 (\Z))} < \infty
,$$when this happens $$\sup_{n \in \N} \norm {P_n  M^*  M  P_n}_{\mathcal{L}(L^2 (\Z))} = \norm{M}_{\mathcal{L}(L^2 (\Z))}^2 .$$
\end{lema}\begin{lema}[Crone]
Let $M$ be a infinite matrix. Then $M$ defines a bounded linear operator on $L^2(\Z)$ if and only if satisfy the following conditions 
\begin{enumerate}
\item[(i)] The rows of $M$ are functions in $L^2 (\Z)$.
\item[(ii)] $(M^*  M)^n$ is defined for all $n \in \N$.
\item[(iii)] $\sup_{n \in \N} \sup_{k \in \Z} 	|(M^*  M)^n_{kk}|^{1/n} < \infty$
\end{enumerate}when this happens we get
$$\sup_{n \in \N} \sup_{k \in \Z} 	|(M^*  M)^n_{kk}|^{1/n} =\norm{M}_{ \mathcal{L}(L^2(\Z))}^2.$$
\end{lema}

Next we will adapt these theorems to periodic pseudo-differential operators. 
\subsection{$L^2$-Boundedness }

Let $\sigma : \T \times \Z \to \C $ be a measurable function such that $\sigma (\cdot , k) \in L^2 (\T)$ for each $k  \in \Z$ and $T_\sigma$ be its associated pseudo-differential operator. Then, for $f \in L^2 (\T)$, we can write
\begin{align*}
    (T_\sigma f) (x) &= \sum_{k \in \Z} \Big( \sum_{m \in \Z} \widehat{\sigma} (m , k) e^{i x \cdot m}\Big) \widehat{f} (k) e^{i x \cdot k} \\
    &= \sum_{j \in \Z} \Big( \sum_{k \in \Z} \widehat{\sigma} (j-k , k) \widehat{f} (k) \Big) e^{i x \cdot j}.
\end{align*}
From this, it is straightforward that the $j$-th Fourier coefficient of $T_\sigma f$ is
$$ \sum_{k \in \Z} \widehat{\sigma} (j-k , k) \widehat{f}(k),$$ which is similar to the $j$-th entry of the matrix-vector product 

$$
    \sum_{k \in \Z} M_{jk} \varphi (k).
$$

This observation is the key fact of this section and is the motivation for the following definition.

\begin{defi}[Associated matrix.]
{\normalfont Let $\sigma : \T \times \Z \to \C $ be a measurable function such that $\sigma (\cdot , k) \in L^2 (\T)$ for each $k  \in \Z$ and $T_\sigma$ be its associated pseudo-differential operator. Then its}  associated matrix {\normalfont $M_\sigma$ is defined as the infinite matrix with entries 

\begin{align*}
    (M_\sigma)_{jk} := \widehat{\sigma} (j-k , k) = \int_\T \sigma (x , k) e^{- i x \cdot (j-k)} dx .
\end{align*}}
\end{defi}

With this definition in mind, the operator $T_\sigma $ can be factored through $L^2 (\Z)$ as the following diagram shows

\begin{center}
\begin{tikzcd}
L^2 (\T) \arrow{r}{T_{\sigma}} \arrow[swap]{d}{\mathcal{F}_{\T}} & L^2 (\T)  \\%
L^2(\Z) \arrow{r}{M_\sigma}& L^2(\Z) \arrow{u}{\mathcal{F}^{-1}_{\T}}
\end{tikzcd}
\end{center}

where $\mathcal{F}_\T$ and $\mathcal{F}_\T^{-1}$ are the toroidal Fourier transform and toroidal inverse Fourier transform, defined by $$\mathcal{F}_\T (f) (k) := \int_\T f(x) e^{-ix \cdot k} dx, \esp \esp \text{ and } \esp \esp \mathcal{F}_\T^{-1} (\widehat{f}) (x) := \sum_{k \in \Z} \widehat{f} (k) e^{i x \cdot k}.$$

These linear operators extend to a unitary operator. For this reason the operator $T_\sigma$ is bounded in $L^2 (\T)$ if and only if the infinite matrix $M_\sigma$ defines a bounded operator in $L^2 (\Z)$ and $\norm{T_\sigma}_{\mathcal{L} (L^2(\T))} = \norm{M_\sigma}_{\mathcal{L} (L^2(\Z))}$ . This allow us to apply Lemmas 3.1 and 3.2 to give necessary and sufficient conditions for $L^2$-Boundedness of pseudo-differential operators. 

\begin{teo}
Let $\sigma : \T \times \Z \to \C $ be a measurable function such that $\sigma (\cdot , k) \in L^2 (\T)$ for each $k  \in \Z$ and $T_\sigma$ be its associated pseudo-differential operator. Then $T_\sigma$ defines a bounded linear operator on $L^2(\T)$ if and only if rows of the associated matrix $M_\sigma$ are in $L^2 (\Z)$, $(M_\sigma^* M_\sigma)^n$ is defined for all $n \in \N$, and 

$$
\sup_{n \in \N} \sup_{k \in \Z} \Big| \int_{\T} \tau^n (x , k) dx \Big|^{1/n}  < \infty
,$$

where $\tau$ is the symbol of $\esp (T_\sigma)^* T_\sigma = T_{\sigma^*}  T_\sigma$ and  $\tau^n$ is the symbol of $(T_\tau)^n$. When this happens then

$$
    \sup_{n \in \N} \sup_{k \in \Z} \Big| \int_{\T} \tau^n (x , k) dx \Big|^{1/n} = \norm{T_\sigma}_{\mathcal{L} (L^2 (\T))}^2.
$$
\end{teo}

\begin{proof}
Let $M_\sigma$ be the associated matrix of $T_\sigma$. Then, since rows of $M_\sigma$ are $L^2(\Z)$, the adjoint matrix is well defined as a densely defined operator in $L^2(\Z)$, and is a pseudo-differential operator with symbol $$\sigma^* (x, j)= \sum_{k \in \Z} \widehat{\sigma} (j - k , k) e^{i x \cdot k}.
$$Now consider the operator $T_{\sigma^*}  T_\sigma$. Since $(M_\sigma^* M_\sigma)^n$ is defined for all $n$, it is a pseudo-differential operator with symbol $\tau(\cdot , k) \in L^2(\T)$ for each $k$ and, again by the existence of $(M_\sigma^* M_\sigma)^n$ for all $n$, $T_\tau^n$ is a pseudo-differential operator with symbol $\tau^n (\cdot , k) \in L^2(\T)$ for each $k$ and for all $n \in \N$. We have already cover satisfy hypothesis $(i)$ and $(ii)$ of Lemma 3.2 and consequently $M_\sigma$ defines a bounded linear operator in $L^2(\Z)$ if and only if aditionally we have $$\sup_{n \in \N} \sup_{k \in \Z} 	|(M_{\sigma^*}   M_\sigma )^n_{kk}|^{1/n} = \sup_{n \in \N} \sup_{k \in \Z} \Big| \int_{\T} \tau^n (x , k) dx \Big|^{1/n}  < \infty.$$
\end{proof}
\begin{coro}
Let $T_\sigma$ be a pseudo-differential operator with symbol $\sigma \in S^m_{\rho , \delta} (\T\times \Z)$. Then $T_\sigma$ defines a bounded operator on $L^2(\T)$ if and only if 
$$
\sup_{n \in \N} \sup_{k \in \Z} \Big| \int_{\T} \tau^n (x , k) dx \Big|^{1/n}  < \infty
,$$

where $\tau$ is the symbol of $\esp (T_\sigma)^* T_\sigma = T_{\sigma^*}  T_\sigma$ and  $\tau^n$ is the symbol of $(T_\tau)^n$. When this happens then

$$
    \sup_{n \in \N} \sup_{k \in \Z} \Big| \int_{\T} \tau^n (x , k) dx \Big|^{1/n} = \norm{T_\sigma}_{\mathcal{L} (L^2 (\T))}^2.
$$
  
\end{coro}
\begin{proof}
The existence of $(M_\sigma^* M_\sigma)^n$ for all $n \in \N$
follows from toroidal composition formula, an the fact that $(M_\sigma^* M_\sigma)^n$ is the matrix associated to the symbol $\tau^n$ of $ ((T_\sigma)^* T_\sigma)^n = (T_{\sigma^*}  T_\sigma )^n$. Also, the $L^2(\Z)$-norm of the $j$-th row of $M_\sigma$ is

\begin{align*}
    \Big( \sum_{k \in \Z} |\widehat{\sigma} (j - k , k)|^2 \Big)^{1/2} &= \Big( \sum_{k \in \Z} |\overline{\widehat{\sigma}^* (k - j , j) }|^2 \Big)^{1/2} \\
    &= \Big( \sum_{k \in \Z} |\widehat{\sigma}^* (k - j , j)|^2 \Big)^{1/2} \\
    &= \norm{\sigma^* (\cdot , j)}_{L^2(\T)} < \infty.
\end{align*}
It follows from Theorem 3.1 that $T_\sigma$ is bounded if and only if 
$$
\sup_{n \in \N} \sup_{k \in \Z} \Big| \int_{\T} \tau^n (x , k) dx \Big|^{1/n}  < \infty
.$$

\end{proof}
\begin{teo}
Let $\sigma : \T \times \Z \to \C $ be a measurable function such that $\sigma (\cdot , k) \in L^2 (\T)$ for each $k  \in \Z$ and $T_\sigma$ be its associated pseudo-differential operator. Then $T_\sigma$ defines a bounded operator on $L^2 (\T)$ if and only if rows of associated matrix $M_\sigma$ are in $L^2 (\Z)$ and  

$$\sup_{n \in \N} \norm{M_{\sigma , n}}_{\mathcal{L}(\ell^2 (\C^{2n+1}))} < \infty,$$

where $M_{\sigma , n}$ is the $(2n+1)\times (2n+1)$ matrix with entries 

\begin{align*}
    (M_{\sigma , n})_{jk} := \overline{\langle \sigma (x,j)e^{i x \cdot j} | \sigma (x,k)e^{i x \cdot k} \rangle}_{L^2 (\T)} \esp , \esp |j|,|k| \leq 2n+1,
\end{align*}

when this happens

$$
\sup_{n \in \N} \norm{M_{\sigma , n}}_{\mathcal{L}(\ell^2 (\C^{2n+1}))} = \norm{T_\sigma}_{\mathcal{L}(L^2 (\T))}^2 .
$$
\end{teo}
\begin{proof}
We just have to see that

\begin{align*}
    (M_{\sigma^*} M_\sigma)_{jk} :&= \sum_{h \in \Z} \widehat{\sigma^*} (j-h,h) \cdot \widehat{\sigma} (h-k, k) \\
    &= \sum_{h \in \Z} \overline{\widehat{\sigma} (h-j,j) } \cdot \widehat{\sigma} (h-k, k) \\
    &= \overline{ \langle \sigma(x , j) e^{i x \cdot j} | \sigma(x,k) e^{i x \cdot k} \rangle}_{L^2(\T)} ,
\end{align*}

and $$\norm {P_n  M_{\sigma^*}  M_\sigma  P_n}_{\mathcal{L}(L^2 (\Z))} = \norm{M_{\sigma , n}}_{\mathcal{L} (\ell^2 (\C^{2n+1}))}.$$

with this the result follows from Lemma 3.1.
\end{proof}

\begin{rem}
Usually, sufficient conditions for boundedness of pseudo-differential operators are given appealing to the existence of a certain number of derivatives in the toroidal variable of the symbol \cite{ruzhansky1, JulioLpbounds, Duvan1}. However, Theorems 3.1 and 3.2 provide necessary and sufficient conditions for boundedness without any derivative. This allows one, for example, to recover the classical result about boundedness of a multiplication operators by a fixed function, which is the same as an operator with associated Toeplitz matrix, using Theorem 3.1.

Let us define $T_\phi f := \phi f$. Then $T_\phi$ is a pseudo-differential operator with symbol $\sigma(x , k)= \phi (x)$, and consequently $T_\phi^*$, $T_\phi^* T_\phi$ and $(T_\phi^* T_\phi)^n$ are pseudo-differential operators with symbols $\overline{\phi}$, $|\phi|^2$ and $|\phi|^{2n}$, respectively. This proves that rows of the associated matrix $M_\sigma$ are in $L^2 (\Z)$, and that $(M_\sigma^* M_\sigma)^n$ is defined for all $n \in \N$ so, $T_\phi$ is bounded if and only if $$
\sup_{n \in \N} \sup_{k \in \Z} \Big| \int_{\T} |\phi|^{2n} (x) dx \Big|^{1/n}  < \infty
,$$  but this equivalent to $$|\phi|^2 \in L^p(\T) \esp \esp \text{for all} \esp 1\leq p < \infty, \esp \esp \text{and} \esp \esp \sup_{p \in [1,\infty)} ||(|\phi|^2)||_{L^p(\T)} < \infty, $$ which is equivalent to $\phi \in L^\infty(\T)$. Another interesting example is given by the pseudo-differential operator associated to the symbol 
\[\sigma(x,k) = \begin{cases}
      0 & \text{if} \esp \esp \frac{1}{2^{|k|}} \leq x \leq 2\pi, \\ 1 & \text{if} \esp \esp 0 \leq x < \frac{1}{2^{|k|}}.
\end{cases}\]
For this operator one has $$\overline{\langle \sigma (x,j)e^{i x \cdot j} | \sigma (x,k)e^{i x \cdot k} \rangle}_{L^2 (\T)} = \int_{0}^{1/2^{\max(|j|, |k|)}}e^{-i x \cdot (j-k)}dx = \frac{-1}{i(j-k)} (e^{-i(j-k)/2^{\max(|j|,|k|)}} - 1),$$ and for $|j|,|k|$ large enough we have the inequality $$|\overline{\langle \sigma (x,j)e^{i x \cdot j} | \sigma (x,k)e^{i x \cdot k} \rangle}_{L^2 (\T)}| = \frac{1}{|j-k|} |e^{-i(j-k)/2^{\max(|j|,|k|)}} - 1| \leq \frac{1}{2^{\max(|j|,|k|)}},$$ proving that the entries of the matrix $M_\sigma^* M_\sigma$ decay rapidly, and thus Theorem 3.2 ensures that $T_\sigma$ defines a bounded operator, even when $\sigma (x,k)$ has no derivatives in the toroidal variable..
\end{rem}
 
\subsection{Gershgorin Theory.}
A beautiful and useful result about the spectrum of a matrix is the Gershgorin's circle theorem which we enunciate below.

\begin{teo}[Gershgorin]
Let $M$ be a $n \times n$ matrix with entries $a_{jk}$ and define $r_j :=  \sum_{ k \neq j} |a_{jk}|$. Then each eigenvalue $\lambda$ of $M$ lies in one of the disks $\overline{B_\C (a_{jj} ,  r_j)}$.
\end{teo}

This theorem can be extended to operators that act on a infinite dimensional space, particularly to infinite matrices. There are many papers on the subject and indeed the Gershgorin theorem gives rise to an entire theory, called the Gershgorin theory. We enunciate below a theorem pertaining to this theory which is a subtle variation Theorem 2 in \cite{SHIVAKUMAR198735}. See also \cite{ALEKSIC2014541}.

\begin{lema}
Let $M$ be a infinite matrix. Define two new matrices $D$ and $F$ by $$D_{jk} := \delta_{jk} M_{jk} ,\esp \esp \text{and } \esp \esp F_{jk}:=(1-\delta_{jk}) M_{jk},$$
where $\delta_{jk}$ is the Kronecker delta. If the following conditions are met

\begin{enumerate}
    \item[(i).] $M_{kk} \neq 0 $ for all $k \in \Z$ and $\inf_{k \in \Z} |M_{kk}| > 0$.
    \item[(ii).] $I + FD^{-1}$ defines a bounded operator in $L^2 (\Z)$ with bounded inverse. 
\end{enumerate}

then $M$ is a invertible densely defined linear operator in $L^2 (\Z)$ with bounded inverse. If in addition

$$
    \lim_{|k| \to \infty} |M_{kk}| =  \infty,
$$

then the inverse of $M$ is a compact operator. 
\end{lema}

\begin{proof}
We just have to write $M = D + F$. If $M$ satisfy condition $(i)$ then $D^{-1}$ exist and defines a bounded operator in $L^2 (\Z)$. Clearly $$M = D + F = (I + FD^{-1}) D, $$ so the inverse of $ M $, if it exists, is $$ M^{-1} = D^{-1} (I + FD^{-1})^{-1}.$$ Applying the hypothesis $(ii)$ we conclude that $ M $ is invertible with continuous inverse. In addition if $$\lim_{|k| \to \infty} |M_{kk}| =  \infty,$$
then $$\lim_{|k| \to \infty} |M_{kk}|^{-1} = 0,
$$ so $D^{-1}$ is a compact operator, and in consequence $M^{-1}$ is compact as well.
\end{proof}

We will now apply this theorem to pseudo-differential operators to provide a way to locate the espectrum. For this the following proposition is needed.

\begin{pro}
Let $M$ be a infinite matrix such that $M$ and $M^*$ both define a bounded operator on $L^\infty (\Z)$. Then $M$ defines a bounded operator in $L^2 (\Z)$ and

$$ \norm{M}_{\mathcal{L} (L^2 (\Z))} \leq \sqrt{\norm{M}_{\mathcal{L} (L^\infty (\Z))} \norm{M^*}_{\mathcal{L} (L^\infty (\Z))}}.$$
\end{pro}

\begin{teo}
Let $\sigma : \T \times \Z \to \C $ be a measurable function such that $\sigma (\cdot , k) \in L^2 (\T)$ for each $k  \in \Z$ and $T_\sigma$ be its associated pseudo-differential operator. If $\sigma$ satisfy the following three properties 

\begin{enumerate}
    \item[(i).]$$\int_\T \sigma (x, k) dx \neq 0\esp \esp \text{for all }\esp k \in \Z \esp, \esp \inf_{k \in \Z }\Big| \int_\T \sigma(x,j) dx \Big| > 0,$$
    \item[(ii).] 
$$\sup_{k \in \Z} \Big| \int_\T \sigma(x,k) dx \Big|^{-1} \cdot \sum_{ j \neq k} |\widehat{\sigma}(j-k,k)| < 1,$$
    \item[(iii).]$$\sup_{j \in \Z} \Big| \int_\T \sigma(x,j) dx \Big|^{-1} \cdot \sum_{k \neq j} |\widehat{\sigma}(j-k,k)| < 1,$$
\end{enumerate}

or equivalently,

\begin{enumerate}
    \item[(i').] $$\int_\T \sigma (x, k) dx \neq 0  \esp \esp \text{for all }\esp k \in \Z  \esp \esp , \esp \inf_{k \in \Z} \Big| \int_\T \sigma(x,j) dx \Big| > 0,$$
    
    \item [(ii').]$$\sup_{k \in \Z} \Big| \int_\T \sigma(x,k) dx \Big|^{-1} \cdot \norm{\widehat{\sigma}(\cdot , k)}_{L^1(\Z)} < 2,$$
    \item[(iii').]$$\sup_{j \in \Z} \Big| \int_\T \sigma(x,j) dx \Big|^{-1} \cdot \norm{\widehat{\sigma^*}(\cdot , j)}_{L^1(\Z)} < 2,$$
\end{enumerate}

where $\sigma^* (x,k)$ is the symbol of adjoint operator, then $T_\sigma$ is an invertible linear operator with bounded inverse. In particular if $$\lim_{|k| \to \infty} \Big| \int_\T \sigma (x,k) dx \Big| = \infty$$ the inverse is a compact operator.
\end{teo}

\begin{proof}
Let $M_\sigma$ be the associated matrix of $T_\sigma$. We will show that this infinite matrix satisfy the hypothesis of Lemma 3.3. First its easy to see that $(i)$ in previous statement is equivalent to $i$ in Lemma 3.3. For the remaining hypothesis define $D_\sigma$ and $F_\sigma$ as
$$ (D_\sigma)_{jk} := \delta_{jk} (M_\sigma)_{jk}, \esp \esp \text{and } \esp \esp (F_\sigma)_{jk}:=(1-\delta_{jk}) (M_\sigma)_{jk},$$then by Proposition 3.1 
\begin{align*}
    \norm{F_\sigma D_\sigma^{-1}}_{\mathcal{L}(L^2(\Z))}^2 &= \norm{I - (I + F_\sigma D_\sigma^{-1} )}_{\mathcal{L}(L^2(\Z))}^2 \\
    &\leq   \norm{F_\sigma D_\sigma^{-1}}_{\mathcal{L} (L^\infty (\T))} \cdot  \norm{(F_\sigma D_\sigma^{-1})^*}_{\mathcal{L} (L^\infty (\T))},
\end{align*}
but clearly 
$$\norm{F_\sigma D_\sigma^{-1}}_{\mathcal{L} (L^\infty (\T))} = \sup_{j \in \Z} \Big| \int_\T \sigma(x,j) dx \Big|^{-1} \cdot \sum_{k \neq j} |\widehat{\sigma}(j-k,k)| < 1,$$
and
$$\norm{(F_\sigma D_\sigma^{-1})^*}_{\mathcal{L} (L^\infty (\T))} = \sup_{k \in \Z} \Big| \int_\T \sigma(x,k) dx \Big|^{-1} \cdot \sum_{ j \neq k} |\widehat{\sigma}(j-k,k)| < 1,$$

so $\norm{F_\sigma D_\sigma^{-1}}_{\mathcal{L}(L^2(\Z))} < 1$ and by Lemma 2.1 in \cite{conwayfa} the operator defined by $I + {F_\sigma}{D_\sigma}^{-1}$ is invertible with bounded inverse in $L^2 (\Z)$.
\end{proof}

\begin{exa}
Let $\alpha : \Z \to \C$ a measurable function that satisfy $\alpha(k) \neq 0$ for all $k$ and $V \in \mathcal{F}_\T^{-1} (L^1 (\Z))$ such that $\int_\T V(x) dx \neq 0$. Then for the symbol $\sigma (x,k) : = \alpha(k) V(x)$ one has

$$\widehat{\sigma} (j - k ,k) = \alpha(k) \widehat{V} (j-k)$$

and from this

$$
\sup_{k \in \Z} \Big| \int_\T \sigma(x,k) dx \Big|^{-1} \cdot \norm{\widehat{\sigma}(\cdot , k)}_{L^1(\Z)} =  \Big| \int_\T V(x) dx \Big|^{-1} \cdot \norm{\widehat{V}}_{L^1(\Z)},
$$
$$\sup_{j \in \Z} \Big| \int_\T \sigma(x,j) dx \Big|^{-1} \cdot \norm{\widehat{\sigma^*}(\cdot , j)}_{L^1(\Z)}  = \sup_{j \in \Z} \Big| \alpha (j) \int_\T V(x) dx \Big|^{-1} \cdot \sum_{k \in \Z} \alpha(k) \widehat{V}(j-k). 
$$
so if $V$ satisfy the conditions $$ \norm{\widehat{V}}_{L^1 (\Z)} < 2 \Big| \int_\T V(x) dx  \Big|,$$ and $$\Big| \alpha (j) \int_\T V(x) dx \Big|^{-1} \cdot \sum_{k \in \Z} \alpha(k) \widehat{V}(j-k) <1 \esp \esp \text{for all} \esp \esp j \in \Z,$$  then $T_\sigma$ is a invertible linear operator with bounded inverse. If a additionally $$\lim_{|k| \to \infty} |\alpha(k) | = \infty,$$ then the inverse is a compact operator.
\end{exa}

\begin{coro}
Let $\sigma : \T \times \Z \to \C $ be a measurable function such that $\sigma (\cdot , k) \in L^2 (\T)$ for each $k  \in \Z$ and $T_\sigma$ be its associated pseudo-differential operator. If the following conditions hold
\begin{enumerate}
    \item[(i).] 
    $$\int_\T \sigma (x, k) dx - \lambda  \neq 0 , \esp \esp \forall k \in \Z \esp \esp , \inf_{k \in \Z} \Big| \int_\T \sigma(x,j) dx - \lambda  \Big| > 0,$$
    \item[(ii).] $$\sup_{k \in \Z} \Big| \int_\T \sigma(x,k) dx - \lambda \Big|^{-1} \cdot \sum_{ j \neq k} |\widehat{\sigma}(j-k,k)| < 1,$$
    \item[(iii).]$$
    \sup_{j \in \Z} \Big| \int_\T \sigma(x,j) dx - \lambda \Big|^{-1} \cdot \sum_{k \neq j} |\widehat{\sigma}(j-k,k)| < 1,$$
\end{enumerate}
then $\lambda \in Res(T_\sigma)$.
\end{coro}

\begin{proof}
We just have to see that for all $\lambda\in \C$ $\esp T_\sigma - \lambda I$ is a pseudo-differential operator with symbol $\sigma(x,k) - \lambda$.
\end{proof}

\begin{exa}
Let $\alpha : \Z \to \C$ a measurable function and $V \in \mathcal{F}_\T^{-1} (L^1 (\Z))$. Then for the symbol $\sigma (x,k) := \alpha(k) + V(x)$ one has 

\[ \widehat{\sigma} (j-k,k) = 
\begin{cases} 
      \widehat{V}(j-k) & j \neq k \\
      \alpha(k) + \int_\T V(x) dx & j=k 
   \end{cases}
\]and then if 
$$\sum_{k \neq j} |\widehat{V}(j-k)| < \inf_{k \in \Z} \Big| \alpha(k) + \int_\T V(x) dx - \lambda \Big|$$ $\lambda \in Res(T_\sigma)$ which proves that $$Spec(T_\sigma) \subseteq \bigcup_{k \in \Z} \overline{B_\C (c_k , r)}$$
where $$c_k = \alpha(k) + \int_\T V(x) dx \esp \esp \text{ and } \esp \esp r =\sum_{ j \neq k} |\widehat{V}(j-k)|$$
\end{exa}
To finish this section we enunciate an improved version of Lemma 3.3, which can be found as the Theorem 3.1 in \cite{FARID19917}. We use it to prove an improved version of Theorem 3.4.

\begin{lema}[Farid and Lancaster]
Let $M$ be an infinite matrix, considered as a linear operator on $L^p (\Z)$ for $1\leq p <\infty$, with columns in $L^1 (\Z)$. Define $r_k := \sum_{j \in \Z , j \neq k} |M_{jk}|$ and assume that 

\begin{enumerate}
    \item[(i).] $M_{kk} \neq 0, \esp \esp \text{for all} \esp k \in \Z$ and $|M_{kk}| \to \infty$ as $|k| \to \infty$.
    \item[(ii).]There exits  $s \in [0,1)$ such that for all $k \in \Z$ $$r_k = s_k |M_{kk}|, \esp \esp s_k \in [0,s].$$
    \item[(iii).] Either $FD^{-1}$ and $(I + \mu FD^{-1})^{-1}$ exists and are in $\mathcal{L}(L^p (\Z))$ for every $\mu \in (0,1]$, or $D^{-1}F$ and $(I + \mu D^{-1} F)^{-1}$ exists and are in $\mathcal{L}(L^p(\T))$.
\end{enumerate}

Then $M$ is a closed operator and the spectrum $Spec(M)$ is nonempty and consists of discrete nonzero eigenvalues, lying in the set 

$$\bigcup_{k \in \Z} \overline{B_\C (M_{kk} , r_k)},$$

where the closed balls $\overline{B_\C (M_{kk} , r_k)}$ are called the Gershgorin discs. Furthermore, any set consisting  of  $n$ Gershgorin discs whose union is disjoint from all other Gersgorin discs intersects $Spec(M)$ in a finite set of eigenvalues of $M$ with
total algebraic multiplicity $n$.
\end{lema}

As an immediate consequence we have:

\begin{teo}
Let $\sigma : \T \times \Z \to \C$ be a measurable function such that $\sigma (\cdot , k) \in L^2(\T)$ for each $k \in \Z$ and $T_\sigma$ be its associated pseudo-differential operator. Let $M_\sigma$ be the associated matrix. Assume that

\begin{enumerate}
    \item[(i).] $\int_\T \sigma(x,k) dx \neq 0,$ for all $ k \in \Z$ and $\big| \int_\T \sigma(x,k) dx \big| \to \infty$ as $|k| \to \infty $.
    \item[(ii).] Rows and columns of $M_\sigma$ are in $L^1 (\Z)$.
    \item[(iii).]$$\sup_{k \in \Z} \big| \int_\T \sigma(x,k) dx \big|^{-1} \cdot \sum_{ j \neq k} |\widehat{\sigma}(j-k,k)| < 1,$$ and$$\sup_{j \in \Z} \Big| \int_\T \sigma(x,j) dx \Big|^{-1} \cdot \sum_{k \neq j} |\widehat{\sigma}(j-k,k)| < 1.$$
\end{enumerate}
Then $T_\sigma$ is a closed operator and the spectrum $Spec(T_\sigma)$ is nonempty and consists of discrete nonzero eigenvalues, lying in the set 

$$\bigcup_{k \in \Z} \overline{B_\C (a_{kk} , r_k)} \esp \esp \text{where} \esp \esp a_{kk} = \int_\T \sigma(x,k) dx, \esp r_k = \sum_{ j \neq k} |\widehat{\sigma}(j-k,k)|.$$

Furthermore, any set of $n$ Gershgorin discs whose union is disjoint from all other Gersgorin discs intersects $Spec(T_\sigma)$ in a finite set of eigenvalues of $T_\sigma$ with
total algebraic multiplicity $n$.
\end{teo}

\begin{proof}
Condition $(i)$ in Theorem 3.5. is the condition $(i)$ of Lemma 3.3. Proposition 3.1. and conditions $(ii)$,$(iii)$ of Theorem 3.5 implies conditions $(ii)$ and $(iii)$ of Lemma 3.3.
\end{proof}

\vspace{3mm}
\section{\textbf{Gohberg's Lemma in $L^p (\T)$}}

In this section we discuss a proof of Gohberg's Lemma in  $L^p (\T)$.

\begin{proof}[ Proof of Gohberg's lemma:]
If $u$ is a nonzero function in $C^\infty (\T)$ then

\begin{align*}
    T_\sigma u (x) &= \sum_{k \in \Z} \sigma (x,k) \widehat{u} (k) e^{i x \cdot k} \\
    &= \sum_{k \in \Z} \int_\T \sigma(x,k) u(y) e^{i(x-y)k} dy \\
    &= \int_\T (\mathcal{F}_\T^{-1} \sigma) (x, x-y) u(y) dy \\
    &=\int_\T (\mathcal{F}_\T^{-1} \sigma) (x, y) u(x-y) dy,
\end{align*}
where $\mathcal{F}_\T^{-1} \sigma$ is the inverse Fourier transform of $\sigma$ with respect to the second variable in a distributional sense. Hence in a distributional sense $$T_\sigma u (x)=\int_\T (\mathcal{F}_\T^{-1} \sigma) (x, y) u(x-y) dy.$$
Since $\sigma$ is continuous it follows that for all $k \in \Z$ exists a $x_k \in [-\pi, \pi]$ such that 
$$|\sigma(x_k , k)| = \sup_{x \in \T} |\sigma(x,k)|.$$
By definition of $d_\sigma$, there exists a subsequence $\{(x_{k_l} , k_l)\}_{l \in \N}$ such that $$|k_l|\to \infty \esp \esp \text{and} \esp \esp |\sigma(x_{k_l} , k_l)|\to d_\sigma,$$
as $l \to \infty$. Then, since $\T$ is compact, the sequence $\{x_{k_l}\}_{l \in \N}$ must have an cluster point $x_0$ in $\T$ and a subsequence $\{x_{k_{l_m}}\}_{m \in \N}$ such that $x_{k_{l_m}} \to x_0$ as $m \to \infty$. For simplicity we will rename the sequence $\{(x_{k_{l_m}} , k_{l_m})\}_{m \in \N}$ as the original one $\{(x_{k_l} , k_l)\}_{l \in \N}$. Define the functions $u_{k_l}$ on $\T$ by $$u_{k_l} (x) = u(x-x_{k_l}) e^{ix \cdot k_l},$$
then wee see that $\norm{u_{k_l}}_{L^p (\T)} = \norm{u}_{L^p (\T)}$ for $l \in \N$. Let's see that the sequence $\{u_{k_l}\}_{l \in \N}$ converges weakly to zero. First, for any $v \in L^q (\T)$
\begin{align*}
    \int_\T u_{k_l} (x) v(x) dx =& \int_\T ( u(x-x_{k_l}) - u(x - x_0) ) v(x) e^{ix \cdot k_l} dx \\
    &+ \int_\T u(x-x_0) v(x) e^{ix \cdot k_l} dx.
\end{align*}

Since by Hölder inequality $u(x-x_0) v(x) \in L^1 (\T)$ then we can apply the Riemann-Lebesgue lemma and consequently, for any $\varepsilon > 0$, there is a natural number $N_1$ such that
$$l \geq N_1 \implies \Big| \int_\T u(x-x_0) v(x) e^{ix \cdot k_l} dx \Big|< \frac{\varepsilon}{2},$$and second, since $u$ is continuous, for any $\varepsilon>0$ there exist $N(\varepsilon) \in \N$, $\delta (\varepsilon) > 0$ such that 
\begin{align*}
    l \geq N(\varepsilon) &\implies |x_{k_l} - x_0| < \delta (\varepsilon) \\
    &\implies |u(x-x_{k_l}) - u(x - x_0)| < \frac{\varepsilon}{2} \cdot \norm{v}_{L^q (\T)}^{-1},
\end{align*}
thus for any $\varepsilon>0$ $$l \geq \max(N_1 , N(\varepsilon)) \implies \Big| \int_\T u_{k_l} (x) v(x) dx \Big|< \varepsilon ,$$
as an immediate consequence $$\norm{K u_{k_l}}_{L^p (\T)} \to 0 \esp \esp \text{as} \esp \esp l \to \infty,$$
for every $K \in \mathfrak{K} (L^p (\T))$. Then for an arbitrary positive number $\varepsilon$ and sufficiently large $l$ 

\begin{align*}
    \norm{K u_{k_l}}_{L^p (\T)} \leq \varepsilon  \norm{u_{k_l}}_{L^p (\T)} \tag{G1}.
\end{align*}

On the other hand, we have the following lemma, whose proof can be found in \cite{Molahajloo2010}.

\begin{lema}
$\norm{\sigma(\cdot,k_l) - T_\sigma u_{k_l}}_{L^p (\T)} \to 0$ as $l \to \infty$.
\end{lema}

Then for sufficiently large $l$,

\begin{align*}
    \norm{\sigma(\cdot,k_l)}_{L^p (\T)} - \norm{ T_\sigma u_{k_l}}_{L^p (\T)} \leq \varepsilon \norm{u}_{L^p (\T)} .\tag{G2}
\end{align*}

By using again the continuity of $\sigma (\cdot , k_l)$ there exists a positive number $\delta$ such that for all $x$ in $[-\pi + x_{k_l}, \pi + x_{k_l}]$ with $|x - x_{k_l}|< \delta$, we have

\begin{align*}
    |\sigma(x, k_l) - \sigma (x_{k_l} , k_l)| < \varepsilon \tag{G3}.
\end{align*}

Let $u \in C^\infty (\T)$ such that $u(x) =0$ for $|x| \geq \delta$. Then $u_{k_l} (x)=0$ for all $x$ in $[-\pi + x_{k_l}, \pi + x_{k_l}]$ with $|x - x_{k_l}| \geq \delta$. So 

\begin{align*}
    \norm{\sigma(x_{k_l},k_l) u_{k_l} }_{L^p (\T)} &- \norm{\sigma(\cdot, k_l) u_{k_l}}_{L^p (\T)} \leq \norm{\sigma(x_{k_l}) u_{k_l} -\sigma(\cdot, k_l) u_{k_l}}_{L^p (\T)}\\
    &= \Big( \int_{- \pi + x_{k_l}}^{\pi + x_{k_l}} |\sigma(x_{k_l},k_l) - \sigma (x,k_l)|^p |u_{k_l} (x)|^p dx \Big)^{1/p}\\
    &= \Big( \int_{|x-x_{k_l}|<\delta} |\sigma(x_{k_l},k_l) - \sigma (x,k_l)|^p |u_{k_l} (x)|^p dx \Big)^{1/p},
\end{align*}then by (G3)\begin{align*}
     | \sigma (x_{k_l} , k_l)| \norm{u}_{L^p (\T)} - \norm{\sigma (\cdot , k_l) u_{k_l}}_{L^p (\T)} < \varepsilon \norm{u}_{L^p(\T)}. \tag{G4}
\end{align*}
Combining (G1), (G2) and (G4), we get for sufficiently large k,

\begin{align*}
    \norm{u}_{L^p (\T)} \norm{T_\sigma - K}_{\mathcal{L} (L^p (\T))} &\geq \norm{(T_\sigma - K)u}_{L^p (\T)}\\
    &\geq \norm{T_\sigma u}_{L^p (\T)} - \norm{K u}_{L^p (\T)}\\
    &\geq \norm{T_\sigma u}_{L^p (\T)} - \varepsilon \norm{ u}_{L^p (\T)}\\
    &\geq \norm{\sigma (\cdot, k_l) u_{k_l}}_{L^p (\T)} - 2 \varepsilon \norm{u}_{L^p (\T)}\\
    & \geq |\sigma (x_{k_l},k_l)| \norm{u}_{L^p (\T)} - 3 \varepsilon \norm{u}_{L^p (\T)}\\
    &= (|\sigma (x_{k_l},k_l)| - 3 \varepsilon) \norm{u}_{L^p (\T)}.
\end{align*}

Letting $l \to \infty$, we get $\norm{T_\sigma - K}_{\mathcal{L} (L^p (\T))} \geq d_\sigma - 3 \varepsilon,$ and finally, using the fact that $\varepsilon$ is an arbitrary positive number, we have $$\norm{T_\sigma - K}_{\mathcal{L} (L^p (\T))} \geq d_\sigma.$$
\end{proof}

\section*{\textbf{Acknowledgments}.}
I sincerely thank the guidance of Carlos Andres Rodriguez Torijano who proposed me this research as the first step in my career as a mathematical researcher. 
I also want to thank professor Michael Ruzhansky for his comments.
\nocite{*}
\bibliographystyle{acm}
\bibliography{main}

\end{document}